\documentclass[12pt]{amsart}
\usepackage{amscd, amssymb, amsfonts}
\usepackage{amsmath,amsthm}
\usepackage{amssymb}

\usepackage{graphicx}

  [2002/01/19 v2.2g %
    AMS font definitions%
  ]
\DeclareFontFamily{U}{euf}{}
\DeclareFontShape{U}{euf}{m}{n}{%
  <5><6><7><8><9>gen*eufm%
  <10><10.95><12><14.4><17.28><20.74><24.88>eufm10%
  }{}
\DeclareFontShape{U}{euf}{b}{n}{%
  <5><6><7><8><9>gen*eufb%
  <10><10.95><12><14.4><17.28><20.74><24.88>eufb10%
  }{}

  [2002/01/19 v2.2g %
    AMS font definitions%
  ]
\DeclareFontFamily{U}{msb}{}
\DeclareFontShape{U}{msb}{m}{n}{%
  <5><6><7><8><9>gen*msbm%
  <10><10.95><12><14.4><17.28><20.74><24.88>msbm10%
  }{}

  [2002/01/19 v2.2g %
    AMS font definitions%
  ]
\DeclareFontFamily{U}{msa}{}
\DeclareFontShape{U}{msa}{m}{n}{%
  <5><6><7><8><9>gen*msam%
  <10><10.95><12><14.4><17.28><20.74><24.88>msam10%
  }{}

\newtheorem{theorem}{Theorem}[section]
\newtheorem{lemma}[theorem]{Lemma}
\newtheorem{proposition}[theorem]{Proposition}
\newtheorem{corollary}[theorem]{Corollary}

\theoremstyle{definition}

\newtheorem{example}[theorem]{Example}

\theoremstyle{remark}
\newtheorem{remark}[theorem]{Remark}

\numberwithin{equation}{section}

\makeatletter
\@namedef{subjclassname@2010}{%
  \textup{2010} Mathematics Subject Classification}
\makeatother

\begin{document}
%%%%% To ease editing, for IMPAN journals add:

\baselineskip=17pt

%%%%%%%%%%%

%% In the running head, replace first names by initials
%% and give an abbreviation of the title.
\title[]
{On the integral of the product of four and more Bernoulli polynomials}

\author{Su Hu}

\address{Department of Mathematical Sciences, Korea Advanced Institute of Science and Technology (KAIST), Daejeon 305-701, South Korea \\Present address: Department of Mathematics and Statistics, McGill University, 805 Sherbrooke St. West, Montr\'eal, Qu\'ebec H3A 2K6, Canada}
\email{husu@kaist.ac.kr, hu@math.mcgill.ca}

\author{Daeyeoul Kim}

\address{National Institute for Mathematical Sciences \\ Doryong-dong \\ Yuseong-gu \\
Daejeon 305-340 \\ South Korea}
\email{daeyeoul@nims.re.kr}

\author{Min-Soo Kim}

\address{Division of Cultural Education, Kyungnam University,
7(Woryeong-dong) kyungnamdaehak-ro, Masanhappo-gu, Changwon-si, Gyeongsangnam-do 631-701, South Korea}
\email{mskim@kyungnam.ac.kr}

%\thanks{*Corresponding author}

\begin{abstract}
In 1958, L.J. Mordell   provided the formula  for the integral of the product of two Bernoulli polynomials, he also remarked: ``The integrals containing the product of more than two Bernoulli polynomials do not appear to lead to simple results."  In this paper, we  provide  explicit
formulas for the integral of the product of $r$ Bernoulli polynomials, where $r$ is any positive integer. Many authors' results in this direction, including N\"orlund, Mordell,
 Carlitz, Agoh and Dilcher are special cases of the formulas given in this paper.
\end{abstract}


%\thanks{Received May 24, 2009}

\maketitle

% \linenumbers


\def\ord{\text{ord}_p}
\def\C{\mathbb C_p}
\def\BZ{\mathbb Z}
\def\Z{\mathbb Z_p}
\def\Q{\mathbb Q_p}
\def\wh{\widehat}
\def\ov{\overline}

\section{Introduction}
\label{Intro}
Nearly 90 years ago, N\"orlund gave the following formula for the integral of the product of two Bernoulli polynomials in his book ``Vorlesungen uber Differenzenrechnung":
\begin{equation}\label{2-in-pro}
\int_0^1B_k(z)B_l(z)dz=(-1)^{k-1}\frac{k!l!}{(k+l)!}B_{k+l}, \; k+l\geq2
\end{equation} (see \cite[p.~31]{Ni}).

Nielsen and Mordell provided two different proofs of (\ref{2-in-pro}) in~\cite{Ni} and~\cite{Mo} respectively. Mordell remarked:`` The integrals containing the product of more than two Bernoulli polynomials do not appear to lead to simple results." (See \cite[p.~375]{Mo}).
But Carlitz \cite{Ca} obtained an explicit formula for the integral of the product of  three and four Bernoulli polynomials which extends (\ref{2-in-pro}).
Later, Wilson \cite{Wi} generalized  Carlitz's result on the integral of the product of  three  Bernoulli polynomials by evaluating  the  integral
\begin{equation}\label{3-in-pro}
\int_0^1\overline{B}_k(az)\overline{B}_l(bz)\overline{B}_m(cz)dz,
\end{equation}
where $\overline{B}_k(x)$ is the periodic extension of $B_k(x)$ on $[0, 1)$ and $a,b,c$
are pairwise coprime integers. Carlitz's result becomes a special case when  $a=b=c=1$. Similar integral evaluations have also been used by Espinosa and Moll \cite{EM}
during their study on Tornheim's double sums.

Recently, Wilson's result has  been generalized by  Agoh and Dilcher \cite{AD2}. In fact, they proved the following result:

\begin{proposition}[{\cite[Proposition 3]{AD2}}]
For $n,m,k\geq 0$, we have \begin{equation}\label{pr3-1}
\frac{I_{n,m,k}(x)}{n!m!k!}=\sum_{a=0}^{n+m}
(-1)^a\sum_{i=0}^a\binom
ai\frac{C_{n-a+i,m-i,k+a+1}(x)}{(n-a+i)!(m-i)!(k+a+1)!}.
\end{equation}
\end{proposition}

Let $r$ be any positive integer. In this paper, we  provide   explicit
formulas for the integral of the product of $r$ Bernoulli polynomials which lets
 all the previously known results  for the integral of the products of Bernoulli polynomials become special cases (see Remark~\ref{es} below).

Before stating our main results, we introduce the definition and some basic properties of  multinomial coefficients.
Since the binomials coefficients have been used by Agoh and Dilcher in \cite{AD2} to get their formulas of the product of  three  Bernoulli polynomials, in order to get more general ones, we need some information on  multinomial coefficients. They are defined to be the number
\begin{equation}\label{mul-def}
\binom{\mu}{k_1,\ldots,k_r}=
\frac{\mu!}{k_1!\cdots k_r!},
\end{equation}
where $k_1+\cdots+k_r=\mu$ and $k_1,\ldots,k_r\geq0$ (see
{\cite[p.~142]{Com}}).

When $r=2$, we obtain the binomial coefficients
\begin{equation}\label{bin-def}
\binom{\mu}{k_1,k_2}=
\binom{\mu}{k_1,\mu-k_1}=\binom{\mu}{k_1}=\binom{\mu}{k_2}.
\end{equation}
The multinomial coefficients can be extended as follows:
\begin{equation}\label{mul-zero}
\binom{\mu}{k_1,\ldots,k_r}=
\begin{cases}
0 & \text{if $\min\limits_{1\leq i\leq r}\{k_i\}<0$ and $\max\limits_{1\leq i\leq r}\{k_i\}>\mu$} \\
\frac{\mu!}{k_1!\cdots k_r!} &\text{otherwise}.
\end{cases}
\end{equation}

\begin{lemma}[{\cite[p.~143]{Com}}]\label{main-lem}
Multinomial coefficients satisfy  the recurrence relation
$$\binom{\mu}{k_1,\ldots,k_r}=\binom{\mu-1}{k_1-1,\ldots,k_r}+\cdots+\binom{\mu-1}{k_1,\ldots,k_r-1}$$
and the symmetry property
$$\binom{\mu}{k_1,\ldots,k_r}=\binom{\mu}{k_{\sigma(1)},\ldots,k_{\sigma(r)}},$$
where $\sigma$ is any permutation of $(1,\ldots,r).$
\end{lemma}

 For any positive integer $m$ and any nonnegative integer
$n,$ using the multinomial coefficients, we may write down the
following multinomial formula which tells us how a sum with $m$
terms expands when raised to an arbitrary power $n$.

\begin{theorem}[{\cite[p.~144]{Com}}] For any $n\geq 0$, $$(x_1 + x_2 + \cdots + x_m)^n =
\sum {n \choose k_1, k_2, \ldots, k_m} \prod_{1\le t\le
m}x_{t}^{k_{t}},$$ where the sum is over all $m$-lists  $(k_1,
k_2, \ldots, k_m)$ of nonnegative integers that sum to $n$.
\end{theorem}

\begin{proposition}\label{main thm}
Let
\begin{equation}\label{three-Eu}
\begin{aligned}
&I_{k_1,\ldots,k_r}(x)=\int_{0}^x B_{k_1}(z)\cdots B_{k_r}(z)dz, \\
&C_{k_1,\ldots,k_r}(x)=B_{k_1}(x)\cdots B_{k_r}(x)-B_{k_1}\cdots B_{k_r},\\
&\widetilde{I}_{k_1,\ldots,k_r}(x)=\frac1{k_1!\cdots k_r!}I_{k_1,\ldots,k_r}(x), \\
&\widetilde{C}_{k_1,\ldots,k_r}(x)=\frac1{k_1!\cdots
k_r!}C_{k_1,\ldots,k_r}(x).
\end{aligned}
\end{equation}
For any $\mu\geq1$ and $k_1,\ldots,k_r\geq0,$ we have
\begin{equation}\label{main}
\begin{aligned}
\widetilde{I}_{k_1,\ldots,k_r}(x) &=\sum_{a=0}^{\mu-1}(-1)^a
\sum_{i_1+\cdots+i_{r-1}=a}\binom a{i_1,\ldots,i_{r-1}} \\
&\qquad\times
\widetilde{C}_{k_1-i_1,\ldots,k_{r-1}-i_{r-1},k_r+a+1}(x) \\
&\quad
+(-1)^\mu
\sum_{i_1+\cdots+i_{r-1}=\mu}\binom \mu{i_1,\ldots,i_{r-1}} \\
&\qquad\times
\widetilde{I}_{k_1-i_1,\ldots,k_{r-1}-i_{r-1},k_r+\mu}(x).
\end{aligned}
\end{equation}
\end{proposition}

Letting
\begin{equation}\label{I-zero}
\widetilde{I}_{k_1,\ldots,k_r}(x)=0\quad\text{when
}\min\limits_{1\leq i\leq r}\{k_i\}<0
\end{equation}
and
$$\mu=k_1+\cdots+k_{r-1}+1$$
in Proposition \ref{main thm}, we obtain the following main result
of this paper.

\begin{theorem}\label{main thm2}
For any $k_1,\ldots,k_r\geq0,$ we have
\begin{equation}\label{main2}\begin{aligned} \widetilde{I}_{k_1,\ldots,k_r}(x)
&=\sum_{a=0}^{k_1+\cdots+k_{r-1}}(-1)^a
\sum_{i_1+\cdots+i_{r-1}=a}\binom a{i_1,\ldots,i_{r-1}}
\widetilde{C}_{k_1-i_1,\ldots,k_{r-1}-i_{r-1},k_r+a+1}(x).
\end{aligned}\end{equation}
\end{theorem}

\begin{remark}\label{es}
In Theorem \ref{main thm2}, when $r=2$ and $x=1$, we obtain the results of  N\"orlund \cite{No} and Mordell \cite{Mo},
when $r=3$ and $x=1$, we obtain the result of Carlitz \cite{Ca}, when $r=3$, we obtain the result of
Agoh and Dilcher \cite{AD2}.
\end{remark}

 From Theorem \ref{main thm2} and the invariance by permutation
of $r$-lists $(k_1,\ldots,k_r)$ in the integral
$\widetilde{I}_{k_1,\ldots,k_r}(x),$ we have the following result.

\begin{corollary}
Let $T_{k_1,\ldots,k_r}(x)$ be the right-hand side of
(\ref{main2}), and let $\sigma\in S_r,$ where $S_r$ is the
symmetric group of degree $r.$ Then for all $k_1,\ldots,k_r\geq0,$
$$T_{k_1,\ldots,k_r}(x)=T_{\sigma(k_1),\ldots,\sigma(k_r)}(x).$$
\end{corollary}
\begin{remark}This is a generalization of \cite[Corollary 2]{AD2}.
\end{remark}

The following parts of this paper are organized as follows.

In section 2, we recall the definitions and some properties for
the Bernoulli polynomials, in particular  well-known facts for the
integral of  one  Bernoulli polynomial. They will be used  in the
proof of the main result of this paper. In section 3, by using our
formulas, we derive several examples including some formulas by
previous authors. In sections 4 and 5, we shall prove Proposition
\ref{main thm} and Theorem \ref{main thm2} respectively.

\section{Bernoulli polynomials}
 The Bernoulli polynomials $B_k(x)$ are defined by the exponential generating function
\begin{equation}\label{Eu-pol}
\frac{te^{xt}}{e^t-1}=\sum_{k=0}^\infty B_k(x)\frac{t^k}{k!}.
\end{equation}
They satisfy the following well-known identities:
\begin{equation}\label{intB-1}
B_k(x+1)-B_{k}(x)=kx^{k-1},\;k\geq1,
\end{equation}
\begin{equation}\label{Eu-pol-1}
\frac{\text{d}}{\text{d}x}B_k(x)=kB_{k-1}(x),\;k\geq1,
\end{equation}
\begin{equation}\label{Eu-pol-2}
\int_x^y B_k(z)dz=\frac{B_{k+1}(y)-B_{k+1}(x)}{k+1},\;k\geq0,
\end{equation}
\begin{equation}\label{Eu-pol-3}
B_k(1-x)=(-1)^kB_{k}(x),\;k\geq0.
\end{equation}
Let $B_k=B_k(0)$ be the Bernoulli numbers. From the above definition of $B_k(x)$,  we conclude that
all the $B_k$ are rational numbers. It is well-known that $B_{2k+1}=0$ for $k\geq1,$ and $B_k$ have alternating signs for even $k.$
Letting $x=0$ in (\ref{Eu-pol-3}), we get the following special value of the Bernoulli
polynomials:
\begin{equation}\label{Eu-p-sp}
B_k(1)=\begin{cases}B_k &\text{for }k\neq1 \\ -B_1 &\text{for }k=1\end{cases}
\end{equation}
(see \cite{AD2}).

\section{Some further consequences}
In this section, we show that Theorem~\ref{main thm2} implies
several results by previous authors.
\begin{proposition}[{\cite[Proposition 1]{AD2}}]
For $k,m\geq 0$ we have
\begin{equation}\label{pr1}\begin{aligned}
\int_{0}^xB_{k}(z)B_{m}(z)dz&=\frac{k!m!}{(k+m+1)!}\sum_{j=0}^{k}(-1)^j
\binom{k+m+1}{k-j} \\
&\quad\times(B_{k-j}(x)B_{m+j+1}(x)-B_{k-j}B_{m+j+1}).
\end{aligned}\end{equation}
\end{proposition}
\begin{proof}By (\ref{three-Eu}), we have
\begin{equation}\label{pr1-1}
 I_{k,m}(x)=\int_{0}^x B_{k}(z)B_{m}(z)dz\end{equation} and
\begin{equation}\label{pr1-2}\widetilde{I}_{k,m}(x)=\frac1{k!m!}I_{k,m}(x).\end{equation}
 Substituting (\ref{pr1-1}) into (\ref{pr1-2}), we have
\begin{equation}\label{pr1-3}\widetilde{I}_{k,m}(x)=\frac1{k!m!}\int_{0}^x
B_{k}(z)B_{m}(z)dz.
\end{equation}
By (\ref{three-Eu}), we  have
\begin{equation}\label{pr1-4}
 C_{k,m}(x)=B_{k}(x)B_{m}(x)-B_{k}B_{m}\end{equation} and
\begin{equation}\label{pr1-5}\widetilde{C}_{k,m}(x)=\frac1{k!m!}C_{k,m}(x).\end{equation}
Substituting (\ref{pr1-4}) into (\ref{pr1-5}), we have
\begin{equation}\label{pr1-6}\widetilde{C}_{k,m}(x)=\frac1{k!m!}(B_{k}(x)B_{m}(x)-B_{k}B_{m}).
\end{equation}
By (\ref{main2}), we have
\begin{equation}\label{pr1-7}
 \widetilde{I}_{k,m}(x)=\sum_{j=0}^{k}(-1)^j
\widetilde{C}_{k-j,m+j+1}(x).\end{equation} Substituting
(\ref{pr1-3}) and (\ref{pr1-6}) into (\ref{pr1-7}), we have
\begin{equation}\begin{aligned}\label{pr1-8}&\quad\frac1{k!m!}\int_{0}^x B_{k}(z)B_{m}(z)dz\\&=\sum_{j=0}^{k}(-1)^j
\frac1{(k-j)!(m+j+1)!}(B_{k-j}(x)B_{m+j+1}(x)-B_{k-j}B_{m+j+1}).
\end{aligned}
\end{equation}
Thus
\begin{equation}\begin{aligned}\label{pr1-9}&\quad\int_{0}^x
B_{k}(z)B_{m}(z)dz\\&=k!m!\sum_{j=0}^{k}(-1)^j
\frac1{(k-j)!(m+j+1)!}(B_{k-j}(x)B_{m+j+1}(x)-B_{k-j}B_{m+j+1})
\\&=\frac{k!m!}{(k+m+1)!}\sum_{j=0}^{k}(-1)^j
\frac{(k+m+1)!}{(k-j)!(m+j+1)!}(B_{k-j}(x)B_{m+j+1}(x)-B_{k-j}B_{m+j+1}).
\end{aligned}\end{equation}
By (\ref{mul-def}) and (\ref{bin-def}), we have
\begin{equation}\label{pr1-10}
\frac{(k+m+1)!}{(k-j)!(m+j+1)!}=\binom{k+m+1}{k-j,m+j+1}=\binom{k+m+1}{k-j}.
\end{equation}
Substituting (\ref{pr1-10}) into  (\ref{pr1-9}), we obtain
(\ref{pr1}).
\end{proof}

\begin{proposition}[{\cite[Proposition 3]{AD2}}]
For $n,m,k\geq 0$ we have
\begin{equation}\label{pr3}
\frac{I_{n,m,k}(x)}{n!m!k!}=\sum_{a=0}^{n+m}
(-1)^a\sum_{i=0}^a\binom
ai\frac{C_{n-a+i,m-i,k+a+1}(x)}{(n-a+i)!(m-i)!(k+a+1)!}.
\end{equation}
\end{proposition}
\begin{proof}By (\ref{three-Eu}), we have
\begin{equation}\label{pr3-1}\widetilde{I}_{n,m,k}(x)=\frac1{n!m!k!}I_{n,m,k}(x)\end{equation}
 and
\begin{equation}\label{pr3-2}\widetilde{C}_{n,m,k}(x)=\frac1{n!m!k!}C_{n,m,k}(x).\end{equation}
By (\ref{main2}), we have
\begin{equation}\label{pr3-3}
 \widetilde{I}_{n,m,k}(x)=\sum_{a=0}^{n+m} (-1)^a\sum_{i=0}^a\binom ai
\widetilde{C}_{n-a+i,m-i,k+a+1}(x).\end{equation} Substituting
(\ref{pr3-1}) and (\ref{pr3-2}) into (\ref{pr3-3}),  we obtain
(\ref{pr1}).
\end{proof}

Letting $r=4$ in Theorem \ref{main thm2}, we have:

\begin{proposition}\label{pro-4}
$$\begin{aligned}
\widetilde{I}_{k_1,k_2,k_3,k_4}(x)
=\sum_{a=0}^{k_1+k_{2}+k_3}(-1)^a \sum_{i_1+i_{2}+i_3=a}\binom
a{i_1,i_{2},i_3}
\widetilde{C}_{k_1-i_1,k_{2}-i_{2},k_3-i_3,k_4+a+1}(x).
\end{aligned}$$
\end{proposition}

Let
$${I}_{k_1,k_2,k_3,k_4}(1)=\int_{0}^1 B_{k_1}(z)B_{k_2}(z)B_{k_3}(z)B_{k_4}(z)dz.$$
Write
$$\widetilde{B}_k=\frac1{k!}B_k.$$
Using (\ref{three-Eu}), (\ref{Eu-pol-3}) and Proposition \ref{pro-4}, we obtain  the following  equalities by integrating  from 0 to 1:
$$\begin{aligned}
\frac{{I}_{k_1,k_2,k_3,k_4}(1)}{k_1!k_2!k_3!k_4!}
&=\sum_{a=0}^{k_1+k_{2}+k_3}(-1)^a
\sum_{i_1+i_{2}+i_3=a}\binom a{i_1,i_{2},i_3} \\
&\quad\times((-1)^{k_1+k_2+k_3+k_4+1}-1) \\
&\quad\times
\widetilde{B}_{k_1-i_1}\widetilde{B}_{k_2-i_2}\widetilde{B}_{k_3-i_3}\widetilde{B}_{k_4+a+1},
\end{aligned}$$
so
\begin{equation}\label{4-int}
\begin{aligned}
\frac{{I}_{k_1,k_2,k_3,k_4}(1)}{k_1!k_2!k_3!k_4!}
&=2\sum_{a=0}^{k_1+k_{2}+k_3}(-1)^{a+1}\widetilde{B}_{k_4+a+1} \\
&\quad\times \sum_{i_1+i_{2}+i_3=a}\binom a{i_1,i_{2},i_3}
\widetilde{B}_{k_1-i_1}\widetilde{B}_{k_2-i_2}\widetilde{B}_{k_3-i_3}
\end{aligned}
\end{equation}
if $k_1+k_2+k_3+k_4\equiv0\pmod2$,
and
\begin{equation}\label{not0}
\frac{{I}_{k_1,k_2,k_3,k_4}(1)}{k_1!k_2!k_3!k_4!}
=0
\end{equation}
if $k_1+k_2+k_3+k_4\not\equiv0\pmod2.$
Therefore, without loss of generality we assume that
\begin{equation}\label{1-4=0}
k_1+k_2+k_3+k_4\equiv0\pmod2\quad\text{and}\quad k_4+a+1\equiv0\pmod2
\end{equation}
since $B_{k_4+a+1}=0$ if $k_4+a+1\not\equiv0\pmod2.$
By (\ref{1-4=0}), we see that
\begin{equation}\label{4-int-con}
\begin{aligned}
k_1+k_2+k_3+k_4+1&=(k_1-i_1)+(k_2-i_2)+(k_3-i_3)+(k_4+a+1)\\
&\equiv1\pmod2.
\end{aligned}
\end{equation}

 Thus from (\ref{4-int}) and (\ref{4-int-con}) we get the following table:

\begin{table}[htbp]
\begin{tabular}{|c||c|c|c|c|} \hline
{} & $k_1-i_1$ & $k_2-i_2$& $k_3-i_3$ &$k_4+a+1$ \\
 \hline
 $\text{(A)} $ & $ \text{odd} $ &$ \text{even} $ &$ \text{even} $ &  $\text{even}$ \\
\hline
 $\text{(B)} $ &$ \text{even} $  & $ \text{odd} $& $\text{even}$ & $\text{even}$ \\
\hline
$\text{(C)} $ &$ \text{even}$  & $\text{even}$& $\text{odd}$ & $\text{even}$  \\
\hline
$\text{(D)} $ &$\text{odd}$  & $\text{odd}$& $\text{odd}$ & $\text{even}$ \\
\hline
 \end{tabular}
\end{table}

From now on, we will use the fact that $B_{2i+1}=0$ for $i\geq1.$
Thus, if $k_j-i_j$ are odd for $j=1,2,3,$ then $k_j-i_j$ must be 1, so we only need to consider the following four cases:
\begin{enumerate}
\item[(A)] We put $i_1=k_1-1.$ Then by (\ref{1-4=0}) and $a=i_1+i_2+i_3,$ we have
$k_2-i_2\equiv k_3-i_3\pmod 2.$

\item[(B)] We put $i_2=k_2-1.$ Similarly, we have
$k_1-i_1\equiv k_3-i_3\pmod 2.$

\item[(C)] We put $i_3=k_3-1.$ Similarly, we have
$k_1-i_1\equiv k_2-i_2\pmod 2.$

\item[(D)] We have
$i_1=k_1-1,i_2=k_2-1,i_3=k_3-1$
since $B_{2i+1}=0$ for $i\geq1.$
\end{enumerate}

Using (\ref{4-int}), (\ref{not0}), (A), (B), (C) and (D), we obtain the following theorem.

\begin{theorem}\label{4-int-reslt}
\begin{enumerate}
\item
If $k_1+k_2+k_3+k_4\equiv0\pmod2,$ then
$$
\begin{aligned}
\frac{{I}_{k_1,k_2,k_3,k_4}(1)}{k_1!k_2!k_3!k_4!}
&=\sum_{a=0}^{k_1+k_{2}+k_3}(-1)^{a}\widetilde{B}_{k_4+a+1} \\
&\quad\times\biggl\{
 \binom{a}{k_1-1}\sum_{i=0}^{a-k_1+1}\binom {a-k_1+1}{i}
 \widetilde{B}_{k_3-i}\widetilde{B}_{k_1+k_2+i-a-1}
 \\
 &\quad\quad+\binom{a}{k_2-1}\sum_{i=0}^{a-k_2+1}\binom {a-k_2+1}{i}
\widetilde{B}_{k_1-i}\widetilde{B}_{k_2+k_3+i-a-1}
 \\
  &\quad\quad+\binom{a}{k_3-1}\sum_{i=0}^{a-k_3+1}\binom {a-k_3+1}{i}
\widetilde{B}_{k_1-i}\widetilde{B}_{k_2+k_3+i-a-1} \biggl\}
\\
&\quad
+\frac12(-1)^{k_1+k_2+k_3}\binom{k_1+k_2+k_3-3}{k_1-1}
\binom{k_2+k_3-2}{k_2-1}\\
&\quad\quad\times\widetilde{B}_{k_1+k_2+k_3+k_4-2}.
\end{aligned}
$$
\item
If $k_1+k_2+k_3+k_4\not\equiv0\pmod2,$ then
$$\frac{{I}_{k_1,k_2,k_3,k_4}(1)}{k_1!k_2!k_3!k_4!}
=0.$$
\end{enumerate}
\end{theorem}

\begin{remark}
Our formula in Theorem \ref{4-int-reslt} is equivalent to (5) in Carlitz \cite[p.~361]{Ca}.
\end{remark}

\begin{example}

Using  Theorem \ref{4-int-reslt}, we immediately get several examples for the integral of the product of four Bernoulli polynomials.

\vskip3mm

 \scriptsize{
\begin{tabular}{l||l} \hline
 $ I_{1,1,1,1}(1) $ & $ \text{(a)} \frac{3}{2}B_1^2B_2 + \frac{1}{4}B_4 - \frac{1}{4}B_2$ \\
\hline

 $ I_{1,1,1,3}(1) $ & $ \text{(a)} \frac{3}{4}B_1^2B_4 + \frac{1}{20}B_6 -
 \frac{1}{8}B_4$ \\
 $  $ & $ \text{(b)} \frac{1}{2}B_2B_4 + \frac{3}{4}B_1^2B_4 + \frac{1}{6}B_6 - \frac{1}{8}B_4$ \\
\hline

 $ I_{1,1,1,5}(1) $ & $\text{(a)}  \frac{1}{2}B_1^2B_6 + \frac{1}{56}B_8 - \frac{1}{12}B_6$ \\
 $  $ & $ \text{(b)}  \frac{5}{6}B_4^2 + \frac{1}{2}B_1^2B_6 + \frac{2}{3}B_2B_6 + \frac{1}{8}B_8 - \frac{1}{12}B_6$ \\
\hline

 $ I_{1,1,2,2}(1) $ & $ \text{(a)} -\frac{1}{6}B_2B_4 - \frac{1}{2}B_1^2B_4 - \frac{1}{15}B_6 + \frac{1}{12}B_4$ \\
 $ $ & $ \text{(b)} \frac{1}{2}B_2^3 + \frac{1}{2}B_2B_4 + B_1^2B_4 + \frac{1}{6}B_6 - \frac{1}{6}B_4$ \\
\hline

 $ I_{1,1,2,4}(1) $ & $ \text{(a)} -\frac{1}{15}B_2B_6 - \frac{1}{5}B_1^2B_6 - \frac{1}{70}B_8 + \frac{1}{30}B_6$ \\
 $  $ & $  \text{(b)}  -\frac{1}{6}B_4^2 - \frac{1}{5}B_1^2B_6 - \frac{1}{5}B_2B_6 - \frac{1}{28}B_8 + \frac{1}{30}B_6$ \\
 $  $ & $  \text{(c)} B_2^2B_4 + \frac{1}{4}B_4^2 + \frac{23}{30}B_2B_6 + \frac{4}{5}B_1^2B_6 + \frac{1}{8}B_8 - \frac{2}{15}B_6$ \\
\hline

 \end{tabular}
}
\end{example}

Letting
$k_1=0$ and $k_2=k,k_3=l,k_4=m$ in Theorem \ref{4-int-reslt}, we obtain the following formula for the integral of the product of three Bernoulli numbers:

\begin{corollary}[{\cite[Corollary 1]{AD2}}] \label{AD-co}For all $k,l,m\geq1,$
$$
I_{k,l,m}(1)=
\begin{cases}
(-1)^{m+1}k!l!m!
\sum_{a=0}^{k+l}\left[\binom a{l-1}+\binom a{k-1}\right]\widetilde{B}_{m+a+1}\widetilde{B}_{k+l-a-1} &\\
\qquad\qquad\qquad\qquad\qquad\qquad\;\text{if }k+l+m\equiv0\pmod{2}
\\
0\qquad\qquad\qquad\qquad\qquad\quad\;\;\;\text{if }k+l+m\not\equiv0\pmod{2}.
\end{cases}
$$
\end{corollary}

\section{Proof of Proposition \ref{main thm}}

Using integration by parts, along with (\ref{three-Eu}) and
(\ref{Eu-pol-1}), we obtain
\begin{equation}\label{part-int}
\begin{aligned}
I_{k_1,\ldots,k_r}(x)&=\frac1{k_r+1}\int_{0}^x B_{k_1}(z)\cdots B_{k_{r-1}}(z)\frac{\text{d}}{\text{d}z}B_{k_r+1}(z)dz \\
&=\frac1{k_r+1}\left[B_{k_1}(z)\cdots B_{k_{r-1}}(z)B_{k_r+1}(z)\right]_0^x \\
&\quad-\frac1{k_r+1}\int_{0}^x \frac{\text{d}}{\text{d}z}(B_{k_1}(z)\cdots B_{k_{r-1}}(z))B_{k_r+1}(z)dz \\
&=\frac1{k_r+1}C_{k_1,\ldots,k_{r-1},k_r+1}(x) \\
&\quad-\frac1{k_r+1}\int_{0}^x
\frac{\text{d}}{\text{d}z}(B_{k_1}(z)\cdots
B_{k_{r-1}}(z))B_{k_r+1}(z)dz.
\end{aligned}
\end{equation}
From (\ref{Eu-pol-1}) we have
\begin{equation}
\begin{aligned}
&\frac{\text{d}}{\text{d}z}(B_{k_1}(z)\cdots B_{k_{r-1}}(z))\\
&\quad=k_1B_{k_1-1}(z)\cdots B_{k_{r-1}}(z)\\
&\quad\quad+\cdots+ k_{r-1}B_{k_1}(z)\cdots B_{k_{r-1}-1}(z).
\end{aligned}
\end{equation}
 From this and (\ref{three-Eu}) we have
\begin{equation}\label{right-hand}
\begin{aligned}
&\frac1{k_r+1}\int_{0}^x \frac{\text{d}}{\text{d}z}(B_{k_1}(z)\cdots B_{k_{r-1}}(z))B_{k_r+1}(z)dz\\
&=\frac{k_1}{k_r+1}
\int_{0}^x B_{k_1-1}(z)\cdots B_{k_{r-1}}(z)B_{k_r+1}(z)dz \\
&\quad+\cdots +\frac{k_{r-1}}{k_r+1}
\int_{0}^x B_{k_1}(z)\cdots B_{k_{r-1}-1}(z)B_{k_r+1}(z)dz \\
&=\frac{k_1}{k_r+1}I_{k_1-1,\ldots,k_{r-1},k_r+1}(x)+\cdots+\frac{k_{r-1}}{k_r+1}I_{k_1,\ldots,k_{r-1}-1,k_r+1}(x),
\end{aligned}
\end{equation}
and substituting  (\ref{right-hand}) into (\ref{part-int}) we get
\begin{equation}\label{part-int-2}
\begin{aligned}
I_{k_1,\ldots,k_r}(x)
&=\frac1{k_r+1}C_{k_1,\ldots,k_{r-1},k_r+1}(x) \\
&\quad-\frac1{k_r+1}(k_1I_{k_1-1,\ldots,k_{r-1},k_r+1}(x)+\cdots+k_{r-1}I_{k_1,\ldots,k_{r-1}-1,k_r+1}(x)).
\end{aligned}
\end{equation}
From (\ref{three-Eu}) and (\ref{part-int-2}), we have
\begin{equation}\label{part-int-3}
\begin{aligned}
\widetilde{I}_{k_1,\ldots,k_r}(x)
=\widetilde{C}_{k_1,\ldots,k_{r-1},k_r+1}(x)
-(\widetilde{I}_{k_1-1,\ldots,k_{r-1},k_r+1}(x)+\cdots+\widetilde{I}_{k_1,\ldots,k_{r-1}-1,k_r+1}(x)).
\end{aligned}
\end{equation}
Now we prove the case when $\mu\geq1$ in Proposition \ref{main
thm} by induction. For $\mu=1$ it just reduces to
(\ref{part-int-3}). Now we assume that Proposition \ref{main thm}
holds for some $\mu\geq1.$ Write
\begin{equation}\label{set-S}
\begin{aligned}
S=\sum_{i_1+\cdots+i_{r-1}=\mu}\binom
\mu{i_1,\ldots,i_{r-1}}\widetilde{I}_{k_1-i_1,\ldots,k_{r-1}-i_{r-1},k_r+\mu}(x).
\end{aligned}
\end{equation}
Equation (\ref{main}) becomes \begin{equation}\label{main1}
\begin{aligned}
\widetilde{I}_{k_1,\ldots,k_r}(x) &=\sum_{a=0}^{\mu-1}(-1)^a
\sum_{i_1+\cdots+i_{r-1}=a}\binom a{i_1,\ldots,i_{r-1}} \\
&\quad\times
\widetilde{C}_{k_1-i_1,\ldots,k_{r-1}-i_{r-1},k_r+a+1}(x)\\
&\quad +(-1)^\mu S.
\end{aligned}\end{equation}
From (\ref{part-int-3}) we have
\begin{equation}\label{part-int-4}
\begin{aligned}
S&=\sum_{i_1+\cdots+i_{r-1}=\mu}\binom \mu{i_1,\ldots,i_{r-1}}
\widetilde{C}_{k_1-i_1,\ldots,k_{r-1}-i_{r-1},k_r+\mu+1}(x) \\
&\quad -\sum_{i_1+\cdots+i_{r-1}=\mu}\binom
\mu{i_1,\ldots,i_{r-1}}
\widetilde{I}_{k_1-(i_1+1),\ldots,k_{r-1}-i_{r-1},k_r+\mu+1}(x)\\
&\quad-\cdots \\
&\quad- \sum_{i_1+\cdots+i_{r-1}=\mu}\binom
\mu{i_1,\ldots,i_{r-1}}
\widetilde{I}_{k_1-i_1,\ldots,k_{r-1}-(i_{r-1}+1),k_r+\mu+1}(x).
\end{aligned}
\end{equation}
Let $i_j'=i_j+1$ in (\ref{part-int-4}) for $j=1,\ldots,r-1.$ From
(\ref{mul-zero}), we obtain
\begin{equation}\label{m-bi-0}
\begin{aligned}
\binom{\mu}{i_1'-1,\ldots,i_{r-1}}= 0 \text{ if
$i_1'=0$},\ldots,\binom{\mu}{i_1,\ldots,i_{r-1}'-1}= 0 \text{ if
$i_{r-1}'=0$},
\end{aligned}
\end{equation}
and
\begin{equation}\label{zero-B}
\begin{aligned}
&\binom{\mu}{i_1'-1,\ldots,i_{r-1}}=
0 \quad\text{if $i_2=\mu+1,\ldots,i_{r-1}=\mu+1$},\\
&\;\ldots,\\
&\binom{\mu}{i_1,\ldots,i_{r-1}'-1}= 0 \quad\text{if
$i_1=\mu+1,\ldots,i_{r-2}=\mu+1$}.
\end{aligned}
\end{equation}
Then we have
\begin{equation}\label{main3}
\begin{aligned}
&\sum_{i_1+\cdots+i_{r-1}=\mu}\binom \mu{i_1,\ldots,i_{r-1}}
\widetilde{I}_{k_1-(i_1+1),\ldots,k_{r-1}-i_{r-1},k_r+\mu+1}(x)
\\&=\sum_{i_1'+\cdots+i_{r-1}=\mu+1}\binom \mu{i_1'-1,\ldots,i_{r-1}}
\widetilde{I}_{k_1-i_1',\ldots,k_{r-1}-i_{r-1},k_r+\mu+1}(x),
\\
&\;\ldots,\\
&\sum_{i_1+\cdots+i_{r-1}=\mu}\binom \mu{i_1,\ldots,i_{r-1}}
\widetilde{I}_{k_1-i_1,\ldots,k_{r-1}-(i_{r-1}+1),k_r+\mu+1}(x)\\&=\sum_{i_1+\cdots+i_{r-1}'=\mu+1}\binom
\mu{i_1,\ldots,i_{r-1}'-1}
\widetilde{I}_{k_1-i_1,\ldots,k_{r-1}-i_{r-1}',k_r+\mu+1}(x).
\end{aligned}
\end{equation}
From (\ref{part-int-4}) and (\ref{main3}) we get
\begin{equation}\label{part-int-5}
\begin{aligned}
S&=\sum_{i_1+\cdots+i_{r-1}=\mu}\binom \mu{i_1,\ldots,i_{r-1}}
\widetilde{C}_{k_1-i_1,\ldots,k_{r-1}-i_{r-1},k_r+\mu+1}(x) \\
&\quad -\sum_{i_1'+\cdots+i_{r-1}=\mu+1}\binom
\mu{i_1'-1,\ldots,i_{r-1}}
\widetilde{I}_{k_1-i_1',\ldots,k_{r-1}-i_{r-1},k_r+\mu+1}(x) \\
&\quad-\cdots \\
&\quad- \sum_{i_1+\cdots+i_{r-1}'=\mu+1}\binom
\mu{i_1,\ldots,i_{r-1}'-1}
\widetilde{I}_{k_1-i_1,\ldots,k_{r-1}-i_{r-1}',k_r+\mu+1}(x).
\end{aligned}
\end{equation}
Let $i_j=i_j'$ for $j=1,\ldots,r-1$ in (\ref{part-int-5}). From
Lemma \ref{main-lem}, (\ref{part-int-5}) becomes
\begin{equation}\label{part-int-6}
\begin{aligned}
S &=\sum_{i_1+\cdots+i_{r-1}=\mu}\binom \mu{i_1,\ldots,i_{r-1}}
\widetilde{C}_{k_1-i_1,\ldots,k_{r-1}-i_{r-1},k_r+\mu+1}(x) \\
&\quad -\sum_{i_1+\cdots+i_{r-1}=\mu+1} \left[\binom
\mu{i_1-1,\ldots,i_{r-1}} +\cdots+\binom \mu{i_1,\ldots,i_{r-1}-1}
\right]
\\
&\quad\times
\widetilde{I}_{k_1-i_1,\ldots,k_{r-1}-i_{r-1},k_r+\mu+1}(x)\\
&=\sum_{i_1+\cdots+i_{r-1}=\mu}\binom \mu{i_1,\ldots,i_{r-1}}
\widetilde{C}_{k_1-i_1,\ldots,k_{r-1}-i_{r-1},k_r+\mu+1}(x) \\
&\quad -\sum_{i_1+\cdots+i_{r-1}=\mu+1}
\binom{\mu+1}{i_1,\ldots,i_{r-1}}
\widetilde{I}_{k_1-i_1,\ldots,k_{r-1}-i_{r-1},k_r+\mu+1}(x).
\end{aligned}
\end{equation}
Substituting (\ref{part-int-6}) into ~(\ref{main1}), we get
Proposition \ref{main thm} when $\mu+1$, which completes our
proof.

\section{Proof of Theorem \ref{main thm2}}

 In order to have $\widetilde{I}_{k_1-i_1,\ldots,k_{r-1}-i_{r-1},k_r+\mu}(x)\neq0$
in Proposition \ref{main thm}, from (\ref{I-zero}) we  require
\begin{equation}\label{con-1}
k_1-i_1\geq0, \ldots,k_{r-1}-i_{r-1}\geq0.
\end{equation}
Thus
\begin{equation}\label{con-2}
\begin{aligned}
&(k_1-i_1)+\cdots+(k_{r-1}-i_{r-1}) \\
&=k_1+\cdots+k_{r-1}-(i_1+\cdots+i_{r-1}) \\
&=k_1+\cdots+k_{r-1}-\mu
\\
&\geq0.
\end{aligned}
\end{equation}
If we let $\mu=k_1+\cdots+k_{r-1}+1$ in (\ref{con-2}), we have
$k_1+\cdots+k_{r-1}-\mu<0,$ which  contradicts   (\ref{con-1}).
Therefore  when $\mu=k_1+\cdots+k_{r-1}+1,$ one of
$$k_1-i_1,\ldots,k_{r-1}-i_{r-1}$$
will always be negative. Thus from  (\ref{I-zero}), all the  terms
$$\widetilde{I}_{k_1-i_1,\ldots,k_{r-1}-i_{r-1},k_1+\cdots+k_{r-1}+k_r+1}(x)$$
vanish. The result then follows from Proposition \ref{main thm}.

\textbf{Acknowledgement} The authors are grateful to the anonymous
referee for his/her helpful comments. This work was supported by
the Kyungnam University Foundation Grant, 2013.

\bibliography{central}

\end{document}